\documentclass[review]{elsarticle}

\usepackage{hyperref}
\usepackage{amsmath}
\usepackage{amssymb}
\usepackage{amsthm}
\usepackage{physics}
\usepackage{color}

\newcommand{\R}{\mathbb{R}}

\newtheorem{theorem}{Theorem}[section]

\newtheorem{proposition}[theorem]{Proposition}

\newtheorem{remark}[theorem]{Remark}

\journal{the European Journal of Control}
\begin{document}
\begin{frontmatter}

\title{Regular path-constrained time-optimal control problems in three-dimensional flow fields\tnoteref{mytitlenote}}
\tnotetext[mytitlenote]{This article is an extended version of the proceeding paper \cite{khaliletal.2018_naples} accepted in the European Control Conference, Naples Italy 25--28 June 2019.}


%\author[mymainaddress,mysecondaryaddress]{R. Chertovskih}
\author[mymainaddress]{R. Chertovskih\corref{visiting}}
\cortext[visiting]{The computational results were performed during a visiting stay at the Federal Research Center ``Computer Science and Control'', Russian Academy of Science, Moscow, Russia}
\author[mymainaddress,mysecondaryaddress]{D. Karamzin}

\author[mymainaddress]{N.T. Khalil\corref{mycorrespondingauthor}}
\cortext[mycorrespondingauthor]{Corresponding author}
\ead{khalil.t.nathalie@gmail.com}
\ead[url]{www.nathaliekhalil.com}

\author[mymainaddress]{F.L. Pereira}

\address[mymainaddress]{Research Center for Systems and Technologies (SYSTEC), Electrical and Computer
	Engineering Department, Faculty of Engineering, University of Porto, 4200-465, Portugal}
\address[mysecondaryaddress]{Federal Research Center ``Computer Science and Control'', Russian Academy of Science, Moscow, Russia}

\begin{abstract}
	
This article concerns a class of time-optimal state constrained control problems with dynamics defined by an ordinary differential equation involving a three-dimensional steady flow vector field. The problem is solved via an indirect method based on the maximum principle in Gamkrelidze's form. The proposed computational method essentially uses a certain regularity condition imposed on the data of the problem. The property of regularity guarantees the continuity of the measure multiplier associated with the state constraint, and ensures the appropriate behavior of the corresponding numerical procedure which, in general, consists in computing the entire field of extremals for the problem in question. Several examples of vector fields are considered to illustrate the computational approach.
\end{abstract}


\end{frontmatter}

\newpage
\section{Introduction}\label{Section_1}
This article concerns the study of time-optimal state-constrained problems and represents an extended version of \cite{khaliletal.2018_naples}. In this article, a particular class of state-constrained time-optimal control problems is considered which satisfies a certain condition of regularity with respect to state constraints. Necessary optimality conditions in the form of the maximum principle and numerical techniques are brought together to shape an indirect method to solve the problem. The fundamental challenge encountered while applying indirect methods in the presence of state constraints is due to the fact that the Lagrange multiplier is a Borel measure whose support is embedded in the set of points in time at which the state trajectory meets the boundary of the state constraint.
The singular component of this measure, more precisely, the atoms, prevent the correct execution of the proposed numerical algorithm. For this reason, in this work, we limit our study to the class of regular problems which satisfy the regularity assumptions with respect to the state constraints proposed in Chapter 6 of the classic monograph~\cite{pontryagin1962mathematical}. This property entails the absence of the atoms, -- that is, the continuity of the measure-multiplier which has been examined, for example, in
\cite{Hager_1979,Maurer_1979,Galbraith_Vinter_2003,Arutyunov_Karamzin_2015,Karamzin_Pereira_2019}.
More precisely, a class of three-dimensional regular state-constrained time-optimal control problems affine with respect to the control and subject to a steady flow field is considered. A similar problem was investigated in \cite{khaliletal.2018_naples} for the case of cylindrical state constraints. Our paper extends the results of \cite{khaliletal.2018_naples} to the case of state constraint in the form of the unit sphere, and of a torus. Although the considered problems are cast in an abstract context, it is not difficult to imagine applications in which they might arise. Precision of the navigation of autonomous vehicles is a key feature for the success of its operation. The shape of the region in which the position error is below a pre-determined  threshold defined in the system requirements depends on the localization method and this might be a cylinder, a sphere, a torus, or even some other more complex shape. Another important case consists in the data sampling underwater currents flowing in the oceanic water column with autonomous underwater vehicles. These underwater currents are, literally, cylinders of water with characteristics quite different from the surrounding water and the motion of the autonomous vehicle should be constrained to this cylinder.

Moreover, our paper extends the analysis in \cite{khalilet.alAUV_2018} where the
two-dimensional case is exploited, to the three-dimensional framework. 
For the proposed optimal control problem, we apply the Pontryagin maximum principle and
derive the associated two-point boundary value problem given with respect to the optimal
control and the measure multiplier. The formulae for the optimal control and for the measure multiplier are derived from
the maximum condition, while the expression for the measure multiplier essentially relies on the regularity. The optimal control and the measure multiplier are expressed in
terms of the state and the adjoint variables. The two-point boundary value problem, which
now depends only on the state variable and on the adjoint variable, being solved, leads to the
determination of the extremals. The continuity of the measure multiplier is the key property
enabling the computation of the junction points of these extremals. 

The three-dimensional case considered here is more complex from computational point of view in comparison to the
two-dimensional case studied in \cite{khalilet.alAUV_2018}. Solution to the corresponding boundary value problem 
requires to find two (instead of one in the two-dimensional case) parameters determining the initial value of 
the adjoint variable. 

Optimal control problems in the presence of state constraints have been widely investigated
in the literature. The classical theory in the field can be found, for example, in
\cite{pontryagin1962mathematical,Dubovitskii_Milyutin_1965,Halkin_1970}
while more recent studies complemented it with such important issues as non-degeneracy and
normality of the maximum principle, non-smooth aspects, non-regular situations, etc., see,
for example,
\cite{Arutyunov_Tynyanskiy_1985,Vinter_Fereira_1994,Arutyunov_Aseev_1997,Arutyunov_2000,Vinter_2000,fontes_2015,Bettiol_Khalil_Vinter_2016},
among others. Computational techniques to solve such problems and to find the set of
extremals have also been broadly studied in the both framework of direct and indirect
methods. Here we refer the reader to the sources
\cite{Bryson_1969,Jacobson,Betts_1993,Fabien,Maurer_2000,Pytlak,Haberkorn_2011,Keulen}
for various issues on numerical solutions to state-constrained problems including direct and
indirect approaches, and to \cite{Jacobson,Fabien,Maurer_2000,Pytlak,Haberkorn_2011,Keulen} where merely indirect numerical methods have been investigated. Numerical methods, closely related to the shooting method, to solve state constrained optimal control problems for nonlinear ODE, were investigated for instance in \cite{malanowski1998sensitivity,bonnans2013shooting}. In \cite{malanowski1998sensitivity}, first-order inequality state constraint are studied. Under some assumptions (for instance strong Legendre-Clebsch condition on the second order derivative of an augmented Hamiltonian, controllability, etc.), which are stronger than ours, the original problems with inequality constraints are locally replaced by some
	auxiliary problems with equality constraints. The classical implicit function theorem is
	applied to the latter problem, and sensitivity results are obtained. The measure multiplier of the original problem is constructed by using the Lagrange multipliers of the auxiliary problem, and in such a way it is continuous. In \cite{bonnans2013shooting}, the analysis is extended to consider second-order state constraint inequalities.

The article is organized as follows. In Section \ref{Section_2}, the problem is formulated
and the regularity concept is presented. In Section \ref{Section_3}, the nondegenerate
maximum principle for the investigated problem is stated. Section \ref{sec:
applications} concerns the application of the maximum principle for the cases in which the
state constraint sets are cylinder, unit sphere, and torus. Explicit formulae for the
extremal control and the measure multiplier are obtained for each case. In section
\ref{Section_4}, we describe the numerical method and illustrate it by showing computational
results for the cylindrical and spherical state constraint cases. Finally, in Section \ref{Section_5}, a brief conclusion and some perspectives of future research are 
given.

\section{Problem formulation}\label{Section_2}

We study a vehicle moving in a three-dimensional bounded and closed state domain defined by the given state constraints. The motion of the vehicle is affected by the presence of the fluid flow vector field $v(x)$, which intervenes in the dynamical control system. The path-constrained time-optimal control problem investigated in this article is as follows:
\begin{equation}
\begin{aligned}\label{problem}
& {\text{Minimize}}
& & T \\
& \text{subject to}
& &  \dot x = u + v(x), \\
&&& x(0)=A,\;\;x(T)=B, \\
&&&  \ g(x) \leq 0  \quad \text{for all } x , \\
&&& h(u):= u_1^2 + u_2^2 + u_3^2 - 1 \le 0.
\end{aligned}
\end{equation}
Here, $x=(x_1,x_2,x_3)$ is the state variable, $u=(u_1,u_2,u_3)$ is the control variable. A measurable function $u(\cdot):[0,T]\to \R^3$ is termed control. The point $A$ is the starting point, while $B$ is the terminal point, and $v:\mathbb{R}^3 \to \mathbb{R}^3$ is a smooth map defining a fluid flow varying in space. The terminal time $T$ is supposed to be minimized.

The state constraint is defined by the given function $g: \R^3 \to \R$. Regarding this function, it is assumed in what follows that $\nabla g(x) \ne 0$ for all $x$ such that $g(x)=0$. Thus, the level set $g(x)=0$ is the so-called regular surface (manifold). Three different cases of the regular surface will be considered in the computational part of this work: cylinder, sphere and torus.

Consider the scalar product of the dynamics and the gradient of the state constraint function:
$$
\Gamma(x,u) = \langle \nabla g(x) , u+v(x) \rangle.
$$
If a feasible control process is considered in the arguments of $\Gamma$, then this function yields the total  time derivative of the state constraint function with respect to the control differential system in study. Following \cite{pontryagin1962mathematical,Arutyunov_Karamzin_2015,Karamzin_Pereira_2019}, consider the a priori regularity condition imposed on the data of problem.

\medskip
\noindent
{\bf Regularity condition.} \label{def:regularity} Assume that for all $x\in\mathbb R^3$ and $u\in \R^3$, such that $g(x)=0$, $\Gamma(x,u)=0$, $h(u)=0$, the set of vectors $\pdv{\Gamma}{u}$ and $\nabla {h(u)}$ is linearly independent.

\medskip
\noindent
\begin{remark}\label{remark:regularity implies controllability} It is simple to show that the regularity condition implies that the controllability conditions w.r.t. the state constraints, also known as IPC (Inward Pointing Condition) and OPC (Output Pointing Condition), are valid along any feasible trajectory (cf. \cite{Arutyunov_2000} and the bibliography cited therein). These conditions intervene in proving the non-degeneracy of the maximum principle.
\end{remark}

The next proposition provides conditions guaranteeing the existence of solution to Problem (\ref{problem}). The existence of solution is important for the forthcoming numerical analysis.

\medskip
\noindent
\begin{proposition} \label{proposition: solution existence}
Assume that the regularity condition holds and that the vector field verifies $|v(x)|< 1$ for all $x=(x_1,x_2,x_3)$ such that $g(x)< 0$. Assume also that $A$, and $B$ belong to the same connected component of the feasible state domain.

Then, Problem (\ref{problem}) has a solution.
\end{proposition}

\begin{proof} Problem (\ref{problem}) is linear while the velocity set is convex and compact. Therefore, in order to ensure the existence of a solution, due to Filippov's theorem, \cite{Filippov_1959}, it is sufficient to justify the existence of at least a single feasible path connecting $A$ and $B$. Consider the connected component of the feasible state domain in which points $A$ and $B$ lie. If points $A$, and $B$ are in the interior of the feasible state domain, then such a path is guaranteed, as $A$ and $B$ can be connected by a smooth curve lying in the interior of the state constraint set, while the full controllability at each point of time is entailed by the condition $|v(x)|< 1$. If one of the point $A$, or $B$, belongs to the boundary, then, by using the regularity of the flow field $v(x)$, which implies the controllability conditions (see Remark \ref{remark:regularity implies controllability} above), there exists a feasible path from $A$ (or $B$) to some close to it point in the interior of the feasible state domain. Then, this case is reduced to the one already considered as the condition $|v(x)|< 1$ once again intervenes at that ``close'' point to guarantee the existence of a feasible path from the starting point $A$ to the terminal point $B$. \end{proof}

\section{Maximum principle}\label{Section_3}

For the maximum principle in Gamkrelidze's form which we consider, the extended Hamilton-Pontryagin function is defined as:
$$
\bar H(x,u,\psi,\mu) = \langle\psi,u + v(x)\rangle - \mu \Gamma(x,u),
$$
where $\psi \in \mathbb R^3$, $\mu \in \mathbb R$ are the adjoint variables.

Assume that the regularity condition holds. Then, for an optimal process $(x^*,u^*,T^*)$, the maximum principle derived in \cite{Arutyunov_Karamzin_2015} (see Theorem 4.5 therein) ensures the existence of Lagrange multipliers: a number $\lambda\in[0,1]$, an absolutely continuous adjoint arc $\psi=(\psi_1,\psi_2, \psi_3) \in W_{1,\infty}([0, T^*]; \mathbb R^3)$, and a scalar continuous function $\mu\in C([0, T^*]; \mathbb R)$, such that the following conditions are satisfied:
\begin{itemize}
	\item[(a)]\label{item: adjoint system}Adjoint equation
	\begin{align*}
	\dot \psi(t) & = -\pdv{\bar H}{x} (x^*(t),u^*(t),\psi(t),\mu(t))  \qquad \text{a.e. } t \in [0,T^*];
	\end{align*}
	\item[(b)]\label{item: max condition} Maximum condition
	\begin{align*}u^*(t) & \in \mathop{\rm argmax}_{u \in U} \{{\bar H}(x^*(t),u,\psi(t),\mu(t)) \} \qquad  \text{a.e. } t\in [0,T^*]
	\end{align*} where $U:= \{ u \in \R^m \ : \ h(u) \leq 0 \}$;
	\item[(c)]\label{item: conservation law} Conservation law
	$$
	\max_{u\in U} \{ \bar H(x^*(t),u,\psi(t),\mu(t)) \} =\lambda \quad \text{for all }\,t\in [0,T^*];
	$$where $U:= \{ u \in \R^m \ : \ h(u) \leq 0 \}$;
	\item[(d)]\label{item: measure continuity} $\mu(t)$ is decreasing, and constant on the time intervals where $g(x^*(t)) < 0$;
	\item[(e)] \label{item: nontriviality condition} Non-triviality condition
	$$
	|\psi(t)-\mu(t)\nabla g(x^*(t))| > 0 \quad \text{ for all }\, t\in [0,T^*].
$$
\end{itemize}

Above, $T^*$ stands for the optimal time, thus, the optimal pair $(x^*,u^*)$ is considered over the time interval $[0,T^*]$.

\section{Applications} \label{sec: applications}

In this part of work, we consider three particular cases of state constraint sets in $\R^3$. For each case, we derive the corresponding adjoint system, and explicit the expressions of the measure multiplier and the optimal control with respect to the state and adjoint variables.

The following simple assertion ensuring the regularity condition under some assumptions on $g$ and $v(x)$ facilitates the analysis of the on-going applications.

\begin{proposition}\label{proposition: regularity general case} Assume that there exists a number $r\in \mathbb R\ne 0$ such that $|\nabla g(x)|=r$ for all $x:$ $g(x)=0$. Then, the regularity condition is satisfied whenever the following condition on the steady flow field $v(x)$ is imposed:
\begin{equation} \label{regularity assumption} |\langle \nabla g(x), v(x)\rangle| < r \;\; \quad \forall\,x:\;g(x)=0.   \end{equation}
\end{proposition}

\begin{proof} It is not restrictive to consider the case $r=1$. Consider any $x,u$ such that $g(x)=h(u)=\Gamma(x,u)=0$. Then,
	$$
	\langle \nabla g(x),u \rangle = - \langle \nabla g(x),v(x) \rangle .
	$$
	By assumption \ref{regularity assumption}, it holds that $| \langle \nabla g(x), u \rangle| < 1$.
	However, both vectors $\frac{1}{2}\nabla h(u)= u$ and $\nabla g(x)$ belong to the unit sphere. Therefore, the set of vectors $u$ and $\pdv{\Gamma}{u} (x,u) = \nabla g(x)$ is linearly independent.
	
\end{proof}

We study three cases of the state constraint set: cylinder, unit sphere, torus. The cylinder case was previously investigated in \cite{khaliletal.2018_naples} and a sample problem was considered for a specific vector field.

\subsection{Cylinder}\label{sec:cylinder}

Consider the case when the function $g$ representing the state constraint takes the following form
\[  g(x) : = x_1^2+x_2^2-1.  \]
Observe that by virtue of Proposition \ref{proposition: regularity general case}
the regularity condition is satisfied if the vector field verifies (\ref{regularity assumption}),  that is, if the vector field verifies the estimate $|\langle x,v(x) \rangle |  < 1$ for all $x$ such that $g(x)=0$.

Next, we explicit the necessary conditions for a given optimal process $(x^*,u^*,T^*)$ to Problem (\ref{problem}). The adjoint system implies (we use the notation $v^*(t)$ for $v(x^*(t))$, and also for its partial derivatives),

\begin{align}
\label{psi1} 
\dot{\psi}_1(t) & = \left(- \psi_1(t) +  2\mu(t) x_1^*(t)\right) \pdv{v_1^*}{x_1} (t)
\nonumber +\left(-\psi_2(t)  + 2\mu(t) x_2^*(t)\right) \pdv{v_2^*}{x_1} (t)  \\ & \qquad
-\psi_3(t) \pdv{v_3^*}{x_1} (t) + 2\mu(t)u_1^*(t) + 2\mu(t) v_1^*(t),
\end{align}

\begin{align}
\label{psi2} 
\dot{\psi}_2(t) &  = \left(- \psi_1(t) + 2\mu(t) x_1^*(t) \right)  \pdv{v_1^*}{x_2} (t) \nonumber   + \left(  - \psi_2(t)  + 2\mu(t) x_2^*(t) \right) \pdv{v_2^*}{x_2} (t)  \\
&\qquad - \psi_3(t)\pdv{v_3^*}{x_2} (t) + 2\mu(t)u_2^*(t) + 2\mu(t) v_2^*(t),
\end{align}

\begin{align}
\label{psi3} 
\dot{\psi}_3(t) & = \left(- \psi_1(t) + 2\mu(t) x_1^*(t) \right)  \pdv{v_1^*}{x_3} (t) \nonumber  + \left(  - \psi_2(t)  + 2\mu(t) x_2^*(t) \right) \pdv{v_2^*}{x_3} (t) \\ & \qquad - \psi_3(t) \pdv{v_3^*}{x_3} (t).
\end{align}

Moreover, the maximum condition and the non-triviality condition allow us to uniquely find the optimal control $(u_1^*,u_2^*,u_3^*)$ (here, the time dependence, for simplicity, is omitted):

\begin{equation}\label{control u_1}
u_1^* = \frac{\psi_1-2\mu x_1^*}{\sqrt{(\psi_1-2\mu x_1^*)^2 + (\psi_2-2\mu x_2^*)^2 +\psi_3^2}}
\end{equation}
\begin{equation}
u_2^* = \frac{\psi_2-2\mu x_2^*}{\sqrt{(\psi_1-2\mu x_1^*)^2 + (\psi_2-2\mu x_2^*)^2 +\psi_3^2}}
\end{equation}
\begin{equation}\label{control}
u_3^* = \frac{\psi_3}{\sqrt{(\psi_1-2\mu x_1^*)^2 + (\psi_2-2\mu x_2^*)^2 +\psi_3^2}}.
\end{equation}
At the boundary of the state constraint, one has
$$
(x_1^*(t))^2 + (x_2^*(t))^2 = 1,\;\; \Gamma(x^*(t),u^*(t))=0,
$$
where
\[ \Gamma(x,u) = 2x_1(u_1+v_1(x)) +2x_2(u_2+v_2(x)) .  \]
Using these relations, replacing $u_1^*, u_2^*$ and $u_3^*$ by their expressions in $\Gamma(x^*,u^*)=0$, we obtain the quadratic equation with respect to $\mu$:

$$
\begin{array}{l}
\mu^2 - \mu (x_1^*\psi_1 + x_2^*\psi_2)
\displaystyle +\,\frac{(x_1^*\psi_1+x_2^*\psi_2)^2-|\psi|^2(x_1^*v_1^*+x_2^*v_2^*)^2}{4(1 - (x_1^*v_1^*+x_2^*v_2^*)^2)} =0.
\end{array}
$$
The solutions are:
$$
\begin{array}{l}
\displaystyle\mu = \frac{1}{2}(x_1^*\psi_1 + x_2^*\psi_2) \pm\,\frac{1}{2} |x_1^*v^*_1+x_2^*v_2^*| \sqrt{\frac{|\psi|^2-(x_1^*\psi_1+x_2^*\psi_2)^2}{1- (x_1^*v^*_1+x_2^*v_2^*)^2}}.
\end{array}
$$
However, the acceptable solution is:
\begin{equation}\label{mu}
\begin{array}{l}
\displaystyle\mu = \frac{1}{2}(x_1^*\psi_1 + x_2^*\psi_2) +\,\frac{1}{2} |x_1^*v^*_1+x_2^*v_2^*| \sqrt{\frac{|\psi|^2-(x_1^*\psi_1+x_2^*\psi_2)^2}{1- (x_1^*v^*_1+x_2^*v_2^*)^2}}.
\end{array}
\end{equation}

The derived expression for $\mu$ holds at the boundary of the state constraints, that is, on the time intervals on which $g(x^*(t))=0$. Moreover, function $\mu(t)$ is continuous, decreasing, and it is constant on the intervals where $g(x^*(t))<0$.

\subsection{Unit sphere}\label{sec:sphere}

For the unit sphere case, the function $g$ is:
\[  g(x) : = x_1^2+x_2^2+x_3^2-1.  \]
Observe that by virtue of Proposition \ref{proposition: regularity general case}
the regularity condition is satisfied if the vector field verifies (\ref{regularity assumption}), that is, if the vector field verifies the estimate $|\langle x,v(x) \rangle |  < 1$ for all $x$ such that $g(x)=0$.

The adjoint system in this case gives:
\begin{align}
\label{spsi1}
\dot{\psi}_1(t) & = \left(- \psi_1(t) +  2\mu(t) x_1^*(t)\right) \pdv{v_1^*}{x_1} (t)  \nonumber +\left(-\psi_2(t)  + 2\mu(t) x_2^*(t)\right) \pdv{v_2^*}{x_1} (t)  \\
& \qquad + \left(-\psi_3(t)+2\mu(t) x_3^*(t)\right) \pdv{v_3^*}{x_1} (t) + 2\mu(t)u_1^*(t) + 2\mu(t) v_1^*(t),\end{align}

\begin{align}
\label{spsi2} 
\dot{\psi}_2(t) &  = \left(- \psi_1(t) + 2\mu(t) x_1^*(t) \right)  \pdv{v_1^*}{x_2} (t) + \left(  - \psi_2(t)  + 2\mu(t) x_2^*(t) \right) \pdv{v_2^*}{x_2} (t) \nonumber \\
&\qquad + \left(-\psi_3(t)+2\mu(t) x_3^*(t)\right)\pdv{v_3^*}{x_2} (t) + 2\mu(t)u_2^*(t) +2\mu(t) v_2^*(t),
\end{align}

\begin{align}
\label{spsi3} 
\dot{\psi}_3(t) & = \left(- \psi_1(t) + 2\mu(t) x_1^*(t) \right)  \pdv{v_1^*}{x_3} (t) + \left(  - \psi_2(t)  + 2\mu(t) x_2^*(t) \right) \pdv{v_2^*}{x_3} (t) \nonumber \\ & \qquad +  \left(-\psi_3(t)+2\mu(t) x_3^*(t)\right)\pdv{v_3^*}{x_3} (t) + 2\mu(t)u_3^*(t) + 2\mu(t) v_3^*(t).
\end{align}

The maximum condition permits to uniquely define the optimal controls:

\begin{equation}\label{control u_1: sphere}
u_1^* = \frac{\psi_1-2\mu x_1^*}{\sqrt{(\psi_1-2\mu x_1^*)^2 + (\psi_2-2\mu x_2^*)^2 +(\psi_3-2\mu x_3^*)^2}}
\end{equation}
\begin{equation}
u_2^* = \frac{\psi_2-2\mu x_2^*}{\sqrt{(\psi_1-2\mu x_1^*)^2 + (\psi_2-2\mu x_2^*)^2 +(\psi_3-2\mu x_3^*)^2}}
\end{equation}
\begin{equation}\label{control sphere}
u_3^* = \frac{\psi_3 - 2\mu x_3^*}{\sqrt{(\psi_1-2\mu x_1^*)^2 + (\psi_2-2\mu x_2^*)^2 +(\psi_3-2\mu x_3^*)^2}}.
\end{equation}
At the boundary of the state constraint, one has
$$
(x_1^*(t))^2 + (x_2^*(t))^2 + (x_3^*(t))^2  = 1,\;\mbox{ and }\; \Gamma(x^*(t),u^*(t))=0,
$$
where
\[ \Gamma(x,u) = 2x_1(u_1+v_1(x)) +2x_2(u_2+v_2(x)) +2x_3(u_3+v_3(x)) .  \]
By following the same reasoning as in the case of the cylinder, and replacing $u_1^*, u_2^*$ and $u_3^*$ by their expressions in $\Gamma(x^*,u^*)=0$, we obtain the following expression of $\mu$:
\begin{align*} & \mu^2 - \mu \langle x^*, \psi \rangle +\frac{ \langle x^*,\psi \rangle ^2-|\psi|^2\langle x^*, v^* \rangle ^2}{4(1 - \langle x^*, v^* \rangle ^2)} =0. \end{align*}

The acceptable solution is:
\begin{align}\label{mu_sphere}
 \mu & =  \frac{1}{2}\langle x^*, \psi \rangle + \frac{1}{2} |\langle x^*, v^* \rangle  | \sqrt{\frac{|\psi|^2-\langle x^*, \psi \rangle  ^2}{1- \langle x^*, v^* \rangle^2}}.
\end{align}

The derived expression for $\mu$ holds at the boundary of the state constraints, that is, on the time intervals on which $g(x^*(t))=0$. Moreover, function $\mu(t)$ is continuous, decreasing, and it is constant on the intervals where \hbox{$g(x^*(t))<0$}.

\subsection{Torus}

A more non-trivial state constraint is represented by the torus symmetric about the $x_3$-axis. In this case, the function $g$ is:
\[ g(x) =  \left(\sqrt{x_1^2+x_2^2} - R\right)^2 + x_3^2 - 1 ,  \]
where $R$ is the so-called major radius.

Observe that by virtue of Proposition \ref{proposition: regularity general case}
the regularity condition is satisfied if the vector field verifies (\ref{regularity assumption}), that is if, $|w x_1v_1 + w x_2v_2 + x_3v_3| < 1$ for all $ x$ such that $g(x)=0$, where $w:= \frac{\sqrt{x_1^2+x_2^2} - R}{\sqrt{x_1^2+x_2^2}}.$

Denote by $w^*:= \frac{\sqrt{(x_1^*)^2+(x_2^*)^2} - R}{\sqrt{(x_1^*)^2+(x_2^*)^2}}$.
The adjoint system in this case gives:
\begin{align*}\label{psi_torus} \dot{\psi}_1(t) & = \left(- \psi_1(t) +  2\mu(t) w^*(t)x_1^*(t)\right) \pdv{v_1^*}{x_1} (t)  \nonumber +\big(-\psi_2(t)  + 2\mu(t) w^*(t)x_2^*(t)\big) \pdv{v_2^*}{x_1} (t)  \\
& \nonumber \qquad \qquad + \left(-\psi_3(t)+2\mu(t) x_3^*(t)\right) \pdv{v_3^*}{x_1} (t)  \\ & + 2\mu(t)\bigg[(u_1^*(t) +  v_1^*(t))w^*(t) + (u_1^*(t)+v^*_1(t))\frac{x^{*2}_1(t) R}{(x^{*2}_1(t)+x^{*2}_2(t))^{\frac{3}{2}}}  \\ & \qquad \qquad   + (u_2^*(t)+v^*_2(t))\frac{x^{*}_1(t)x^{*}_2(t) R}{(x^{*2}_1(t)+x^{*2}_2(t))^{\frac{3}{2}}}     \bigg], \nonumber \end{align*}

\begin{align*}
\dot{\psi}_2(t) &  =  \left(- \psi_1(t) +  2\mu(t) w^*(t)x_1^*(t)\right) \pdv{v_1^*}{x_2} (t)  \nonumber +\left(-\psi_2(t)  + 2\mu(t) w^*(t)x_2^*(t)\right) \pdv{v_2^*}{x_2} (t)  \\
& \nonumber \qquad \qquad + \left(-\psi_3(t)+2\mu(t) x_3^*(t)\right) \pdv{v_3^*}{x_2} (t)  \\ & + 2\mu(t)\bigg[(u_2^*(t) +  v_2^*(t))w^*(t) + (u_2^*(t)+v^*_2(t))\frac{x^{*2}_2(t) R}{(x^{*2}_1(t)+x^{*2}_2(t))^{\frac{3}{2}}}  \\ & \qquad \qquad   + (u_1^*(t)+v^*_1(t))\frac{x^{*}_1(t)x^{*}_2(t) R}{(x^{*2}_1(t)+x^{*2}_2(t))^{\frac{3}{2}}}  \bigg], \nonumber
\end{align*}

\begin{align*}
\dot{\psi}_3(t) & = \left(- \psi_1(t) + 2\mu(t)w^*(t) x_1^*(t) \right) \pdv{v_1^*}{x_3} (t) \nonumber  + \left(  - \psi_2(t)  + 2\mu(t)w^*(t) x_2^*(t) \right) \pdv{v_2^*}{x_3} (t) \\ & \qquad +  \left(-\psi_3(t)+2\mu(t) x_3^*(t)\right)\pdv{v_3^*}{x_3} (t) + 2\mu(t)u_3^*(t) + 2\mu(t) v_3^*(t).
\end{align*}

From the maximum condition, we obtain the optimal controls:

\begin{equation*}\label{control u_1: torus}
u_1^* = \frac{\psi_1-2\mu w^* x_1^*}{\sqrt{(\psi_1-2\mu w^* x_1^*)^2 + (\psi_2-2\mu w^* x_2^*)^2 +(\psi_3-2\mu x_3^*)^2}}
\end{equation*}
\begin{equation*}\label{control u_2: torus}
u_2^* = \frac{\psi_2-2\mu  w^*x_2^*}{\sqrt{(\psi_1-2\mu  w^*x_1^*)^2 + (\psi_2-2\mu  w^*x_2^*)^2 +(\psi_3-2\mu x_3^*)^2}}
\end{equation*}
\begin{equation*}\label{control u_3: torus}
u_3^* = \frac{\psi_3 - 2\mu x_3^*}{\sqrt{(\psi_1-2\mu w^* x_1^*)^2 + (\psi_2-2\mu w^* x_2^*)^2 +(\psi_3-2\mu x_3^*)^2}}.
\end{equation*}

We have
\[ \Gamma(x,u) = 2wx_1(u_1+v_1)+2wx_2(u_2+v_2)+ 2x_3(u_3+v_3) ,  \]
where $\displaystyle w:= \frac{\sqrt{x_1^2+x_2^2} - R}{\sqrt{x_1^2+x_2^2}}.$ Therefore, at the boundary of the state constraint, the following quadratic equation with respect to $\mu$ arises:
\begin{align*}
\mu^2 - & \mu  (w^*x_1\psi_1+w^*x_2\psi_2+x_3\psi_3)  \\ & + \frac{(w^*x_1\psi_1+w^*x_2\psi_2+x_3\psi_3) ^2 - |\psi|^2 (w^*x_1v_1+w^*x_2v_2+x_3v_3)^2}{4 (1 - (w^*x_1v_1+w^*x_2v_2+x_3v_3) ^2)} =0.
\end{align*}

The acceptable solution is
\begin{align*}
\mu = &  \frac{1}{2}(w^*x_1^*\psi_1 + w^*x_2^*\psi_2 + x_3^*\psi_3) \\ & + \frac{1}{2} |w^*x_1^*v^*_1+w^*x_2^*v_2^*+x_3^*v_3^*| \sqrt{\frac{|\psi|^2-(w^*x_1^*\psi_1+w^*x_2^*\psi_2+x_3^*\psi_3)^2}{1- (w^*x_1^*v^*_1+w^*x_2^*v_2^*+x_3^*v_3^*)^2}}.
\end{align*}
The derived expression for $\mu$ holds at the boundary of the torus. Moreover, function $\mu(t)$ is continuous, decreasing, and it is constant on the intervals where $x^*(t)$ is within the torus interior.
Note that when $R=0$, we recover the formulae in the case of the unit sphere.

\section{Numerical results}\label{Section_4}

The proposed computational algorithm is based on the maximum principle, that is, on the conditions (a)--(e) stated in section \ref{Section_3}. As it has been shown in the previous section, the regularity condition enabled us to obtain the expression for the measure multiplier through the state and adjoint variables $x(t)$ and $\psi(t)$ respectively. Upon substitution of this expression as well as the ones for extremal controls into the adjoint system, a two-point boundary-value problem is obtained:
$$
\dot{x}=u^*+v(x),\quad x(0)=A,\ x(T^*)=B,
$$
together with (\ref{psi1})-(\ref{mu}) for the cylinder or (\ref{spsi1})-(\ref{mu_sphere}) for the unit sphere case. Here, the fluid flow $v(x)$, the starting point, $A$, and the terminal point, $B$, are given, and $T^*$, the optimal travelling time, is unknown. The problem is solved numerically by a variant of the shooting method described below. See for instance \cite{nr,bvp} for an overview on the numerical methods for two-points boundary-value problems.

The measure multiplier $\mu(t)$ is non-constant only when the corresponding trajectory $x(t)$ lies at the boundary of the state constraint. The continuity of the measure multiplier is used for computation of the junction points, that is the points where the extremal arc meets the boundary of the state constraint.

Let us briefly outline the numerical algorithm used for the computation of the field of extremals.

We set $\mu(0)=0$. In this regard, see Remark 3.1 in \cite{jota}. Then, condition (e) featured in section  \ref{Section_3} enables us to consider the initial value for $\psi$ from the surface of the unit sphere, $|\psi(0)|=1$. This unit sphere surface is parameterized by the two angles $\theta\in[0,\pi]$ and $\phi\in[0,2\pi)$:
\begin{align*} \psi_1(0)&=\sin\theta\cos\phi, \quad 
\psi_2(0)=\sin\theta\sin\phi, \quad  \psi_3(0)=\cos\theta. \end{align*}

For a given value of $\theta$ and $\phi$, the system of governing equations is integrated numerically by the fourth-order Runge-Kutta method. Using the bisection method in $\theta$, and $\phi$, the trajectories not meeting the boundary and satisfying \hbox{$|x(T^*)-B|<10^{-3}$} are computed, being $10^{-3}$ the required accuracy.

The bisection method is also used to compute the junction points of the trajectories meeting
the boundary, but only those trajectories for which $\mu$ is continuous at the junction
point, i.e. $|\mu|<10^{-3}$, are selected. For such trajectories, the equations governing
the overall system are integrated further in time, and, according to (\ref{control
u_1})--(\ref{mu}) for the cylinder and to (\ref{control u_1: sphere})--(\ref{mu_sphere}) for
the unit sphere case, the resulting trajectory follows the boundary and never leave it. Integrating the system along the boundary, at each time step, $\bar{t}$, we also compute trajectories ``leaving'' the boundary -- corresponding to the initial conditions $x=x(\bar{t}),$ $\psi=\psi(\bar{t})$ and assuming $\mu=\mu(\bar{t})$ to be constant for all $t>\bar{t}$. If such trajectory satisfies \hbox{$|x(T^*)-B|<10^{-3}$} for $T^*>\bar{t}$, it belongs to an extremal, together with the corresponding boundary segment and the trajectory entering the boundary. 

In the first example, we apply this numerical procedure to the Problem~(\ref{problem}) with
the state constraints given by the cylinder (see section~\ref{sec:cylinder}) for the fluid flow
$v(x)=(0,\,0,\,x_1^2+x_2^2)$ satisfying the regularity condition (\ref{regularity
assumption}), \hbox{$A=(0.2,\,-0.5,\,0)$} and \hbox{$B=(0,\,0.5,\,5)$}. This vector field $v(x)$, which corresponds to a fluid flowing faster on the boundary of the cylinder, is chosen to obtain extremals with active boundary. Indeed, travelling along the boundary is more beneficial, in the sense that such trajectories
take less time to join their endpoints. The set of extremals is shown in the space of
$(x_1,\,x_2,\,x_3)$ in Figure~\ref{f1}, as well as its projection onto the plane $x_3=0$. 
The set is constituted of two extremals -- one containing a boundary segment (red line), 
and one not meeting the boundary (black line).

\begin{figure}[t]
\begin{minipage}[t]{0.98\linewidth}
	\includegraphics[scale=.9]{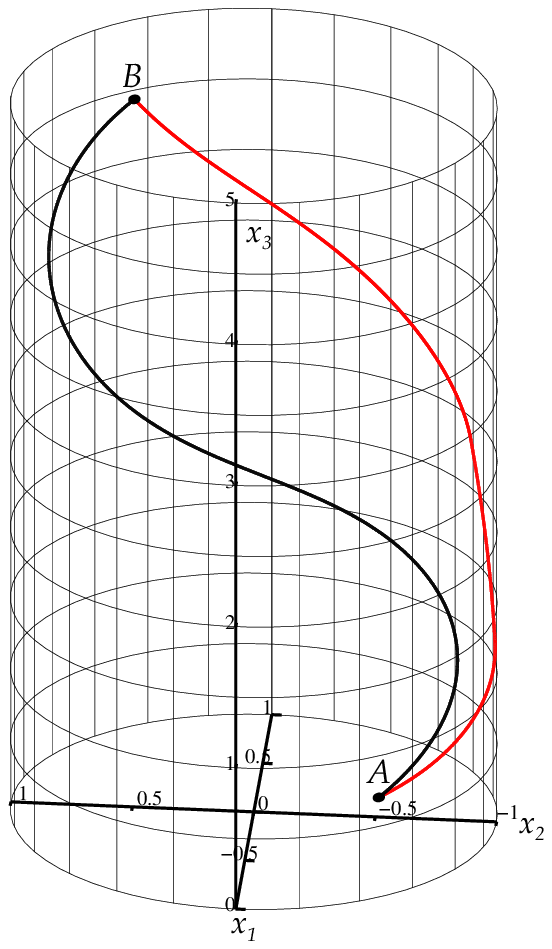} \quad \includegraphics[scale=1.1]{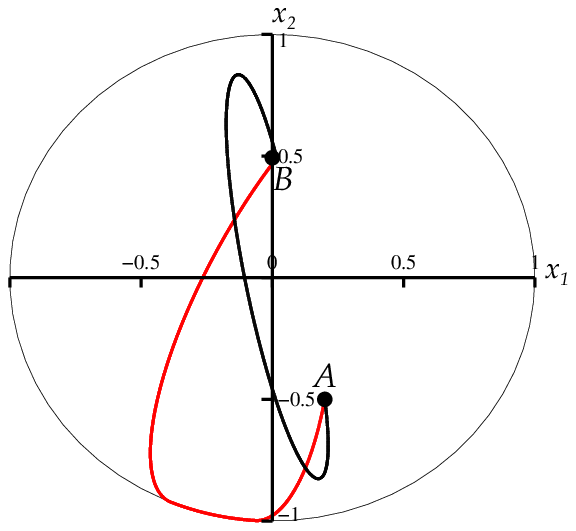}
\end{minipage}%
\caption{Left figure: set of extremals for $v(x)=(0,\,0,\,x_1^2+x_2^2)$, $A=(0.2\,-0.5,\,0)$
	and $B=(0,\,0.5,\,5)$. Right figure: projection of the set of extremals on the plane $x_3=0$.\label{f1}}
\end{figure}

\begin{figure}[t]
\begin{minipage}[t]{1.2\linewidth}
 \includegraphics[scale=.95]{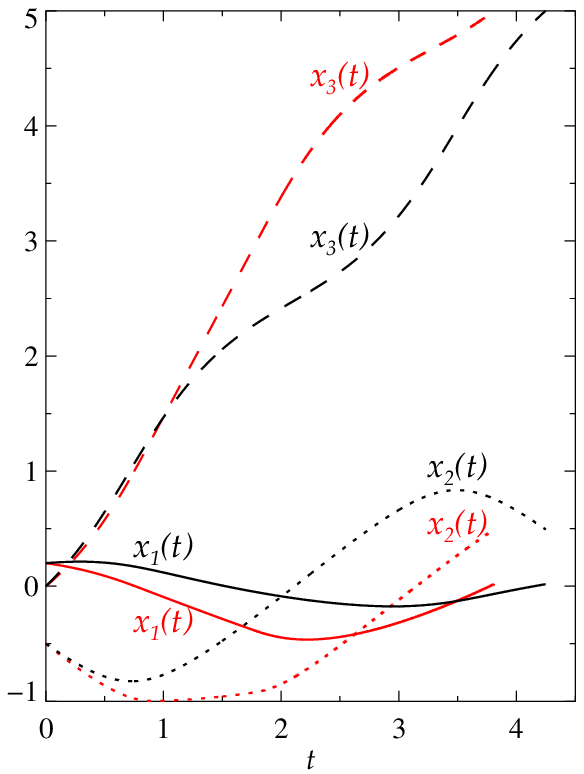} \quad
	\includegraphics[scale=.95]{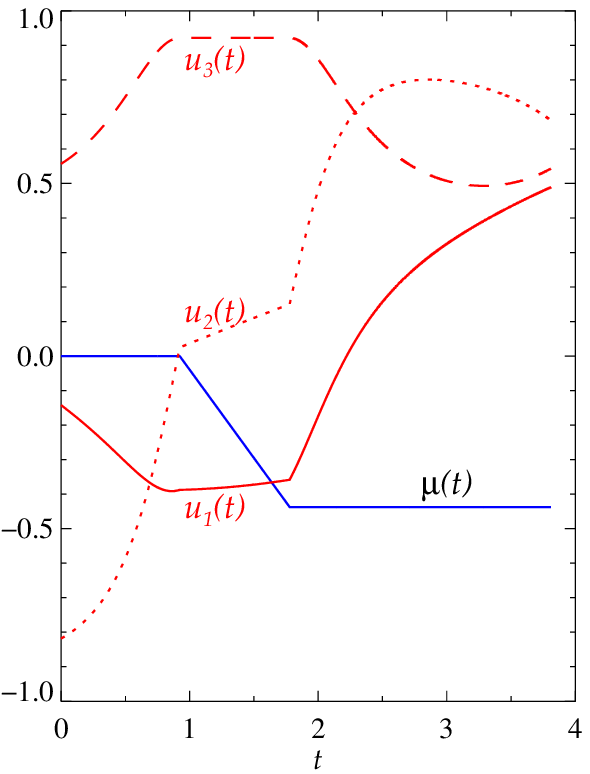}
\end{minipage}%
\caption{Left figure: evolution of each component of the extremals shown in Figure~\ref{f1}. 
The components labeled by black (red) correspond to the extremal not meeting the boundary
(possessing a boundary segment). Right figure: evolution of each component of the control 
variables $u_1^*(t)$ (red solid), $u_2^*(t)$ (red dotted) and $u_3^*(t)$ (red dashed) 
as well as of the Lagrange multiplier $\mu(t)$ (blue solid) corresponding to the optimal 
extremal shown in red in Figure~\ref{f1}.\label{f2}}
\end{figure}

The optimal extremal (red line) has a boundary segment (for
\hbox{$0.92\le t\le 1.78$}), the corresponding travelling time is $T^*=3.81$; the path from
$A$ to $B$ along the extremal not meeting the boundary (black line) takes $4.25$ time units. Time
evolution of each component of both extremals is shown in Figure~\ref{f2}. For the optimal
extremal, the evolution of the control, $u^*(t)$, as well as of the Lagrange multiplier
$\mu(t)$ is also shown in Figure ~\ref{f2}.

By construction, the multiplier $\mu(t)$ is constant when the trajectory is in the interior
of the state constraint (i.e. for $t<0.92$ and $t>1.78$). Along the boundary (see right
panel of Figure~\ref{f2}), the measure multiplier demonstrates linear behavior with respect to time. 
This can be shown analytically for the considered fluid flow.

Next, we consider the spherical case described in section~\ref{sec:sphere}. 
We take the following flow representing a horizontal vortex,
\begin{equation}
\label{sv}
v(x)=\left( \dfrac{4}{1+{\rm e}^{-6x_2}}-2,  -\dfrac{4}{1+{\rm e}^{-6x_1}}+2,  0 \right),
\end{equation}
satisfying the regularity condition (\ref{regularity assumption}).
The starting and terminal points are chosen to be $A=(0.6,\,0.6,\,0.4)$ and $B=(-0.6,\,-0.6,\,0)$, 
respectively. The set of computed extremals, shown in Figure~\ref{fs}, is represented by three
curves, two of them do not meet the boundary (shown in black and red) and one meeting the boundary 
(shown in blue). Evolution of coordinates of the extremals not meeting the boundary as well as 
controls for the optimal extremal are shown in Figure~\ref{contrs}.

Travelling along the extremal shown in black takes 1.73, while travelling along the optimal
extremal (shown in red) is roughly twice faster (0.81 time units). Although the flow is
faster closer to the boundary, travelling along it is not favorable (as in the previous
example), it takes 1.98 time units.

\begin{figure}[t]
	\begin{minipage}[t]{1.5\linewidth}\hskip-14ex
		\includegraphics[scale=1.4]{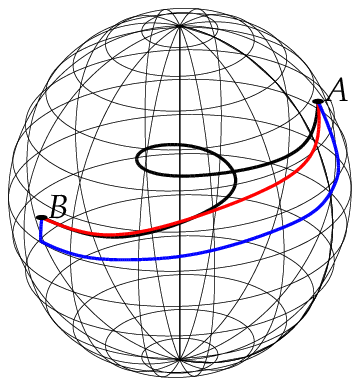} \hskip-6ex \includegraphics[scale=1.2]{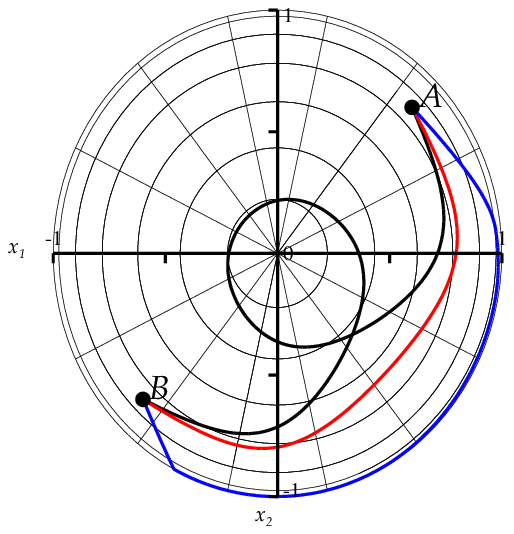}
	\end{minipage}%
	\caption{Set of extremals for the unit sphere case, the flow (\ref{sv}), $A=(0.6,\,0.6,\,0.4)$ and $B=(-0.6,\,-0.6,\,0)$. 
		Projection of the set of extremals on the plane $x_3=0$ is shown in the figure on the right.\label{fs}}
\end{figure}

%\begin{figure}[t!]
%\centerline{\includegraphics[scale=1.7]{sphere1}}
%\vspace*{-1cm}
%\centerline{\includegraphics[scale=1.2]{sphere2}}
%\caption{Set of extremals for the unit sphere case, the flow (\ref{sv}), $A=(0.6,\,0.6,\,0.4)$ and $B=(-0.6,\,-0.6,\,0)$. 
%Projection of the set of extremals on the plane $x_3=0$ is shown in the bottom figure.\label{fs}}
%\end{figure}

\begin{figure}[t]
\centerline{\includegraphics[scale=.95]{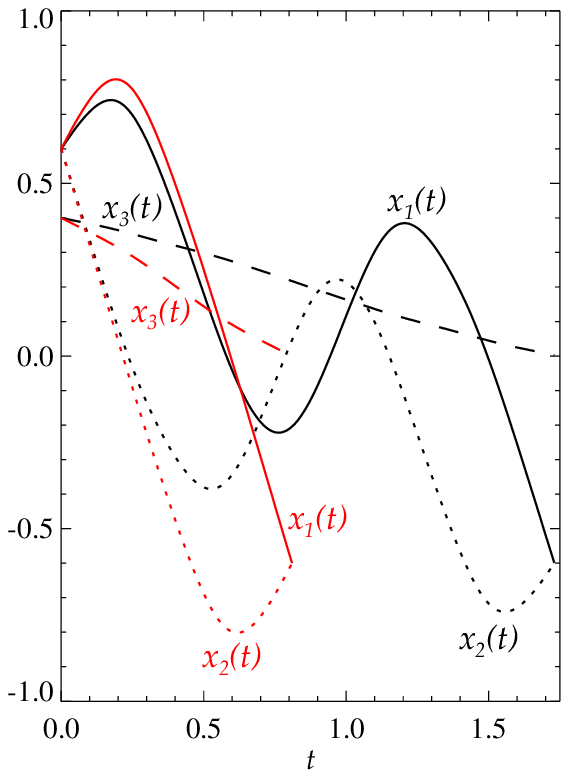} \quad \includegraphics[scale=.95]{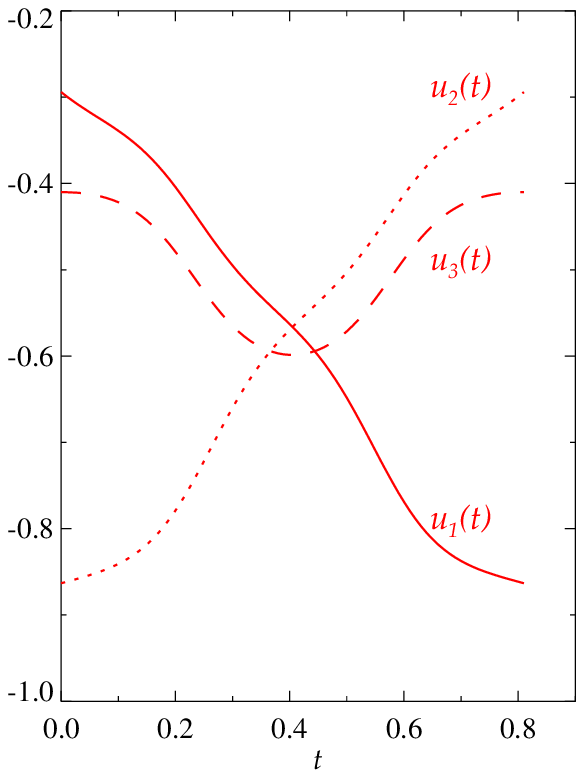}}
\caption{Left figure: evolution of each component of the extremals not meeting the boundary 
shown in Figure~\ref{fs}. Right figure: evolution of each component of the control 
variables $u_1^*(t)$ (red solid), $u_2^*(t)$ (red dotted) and
$u_3^*(t)$ (red dashed) corresponding to the optimal extremal.\label{contrs}}
\end{figure}

\section{Conclusions}\label{Section_5}

In this paper, a time-optimal control problem in a three-dimensional steady vector
flow field is analyzed in the presence of a state constraint given in the form of cylinder,
unit sphere, and torus. Regularity conditions with respect to the state constraints and the vector
flow field in study are considered. The maximum principle is applied to the problem in
question under these regularity conditions. It is shown how the regularity condition, which
implies the continuity of the measure multiplier, assists in solving numerically the
boundary-value problem associated to the maximum principle. In this context, explicit
formulae for the measure multiplier and the extremal controls are obtained as expressions
with respect to the state and adjoint functions. These expressions are further substituted into the
boundary-value problem, and solved by a variant of the shooting method. The obtained results are
numerically illustrated for several three-dimensional sample vector flow fields and state
constraints, and the corresponding set of extremals is plotted.

There are several avenues to extend this work. One possible way may consist in decreasing the degree of smoothness of the mapping defining the state constraints. Note that the proposed computational method essentially relies on the existence of the second order derivative of the mapping $g$. At the same time, an appropriate approximation sequence of smooth problems may assist in solving numerically the control problem in which $g$ is of merely $C^1$-class. Another avenue may consist in considering vector-valued state constraints. There are a number of possibilities to address this problem that the research effort should consider. Still, another direction is the consideration of general nonlinear dynamics. In this case, we might study problems for which the current framework holds and problems for which regularization techniques have to be applied. Finally, here the vector field mapping is defined a priori. There are interesting problems, for which this not the case, and the control action may affect the  evolution of the vector field mapping. This latter problem is of significant complexity but also relevant for many applications.

\section*{Acknowledgements}

The first two authors are supported by the Russian Science Foundation during the project
19-11-00258 carried out in the Federal Research Center ``Informatics and Control'' of the
Russian Academy of Sciences.  The valuable support of FCT (Portugal), research funding
granted to the SYSTEC R\&D Unit under project NORTE-01-0145-FEDER-000033 -- STRIDE (COMPETE
2020), is also highly acknowledged.

\section*{References}

\bibliography{mybibfile}

\begin{thebibliography}{10}
\expandafter\ifx\csname url\endcsname\relax
  \def\url#1{\texttt{#1}}\fi
\expandafter\ifx\csname urlprefix\endcsname\relax\def\urlprefix{URL }\fi
\expandafter\ifx\csname href\endcsname\relax
  \def\href#1#2{#2} \def\path#1{#1}\fi

\bibitem{khaliletal.2018_naples}
R.~Chertovskih, D.~Karamzin, N.~T. Khalil, F.~L. Pereira, Path-constrained
  trajectory time-optimization in a three-dimensional steady flow field, in:
  2019 European Control Conference, IEEE, 2019.

\bibitem{pontryagin1962mathematical}
L.~S. Pontryagin, V.~Boltyanskii, R.~Gamkrelidze, E.~Mishchenko, {The
  mathematical theory of optimal processes, translated by KN Trirogoff}, New
  York, {1962}.

\bibitem{Hager_1979}
W.~W. Hager, Lipschitz continuity for constrained processes, SIAM Journal on
  Control and Optimization 17~(3) (1979) 321--338.

\bibitem{Maurer_1979}
H.~Maurer, Differential stability in optimal control problems, Applied
  Mathematics and Optimization 5~(1) (1979) 283--295.

\bibitem{Galbraith_Vinter_2003}
G.~N. Galbraith, R.~B. Vinter, Lipschitz continuity of optimal controls for
  state constrained problems, SIAM journal on control and optimization 42~(5)
  (2003) 1727--1744.

\bibitem{Arutyunov_Karamzin_2015}
A.~V. Arutyunov, D.~Y. Karamzin, On some continuity properties of the measure
  lagrange multiplier from the maximum principle for state constrained
  problems, SIAM Journal on Control and Optimization 53~(4) (2015) 2514--2540.

\bibitem{Karamzin_Pereira_2019}
D.~Karamzin, F.~L. Pereira, On a few questions regarding the study of
  state-constrained problems in optimal control, Journal of Optimization Theory
  and Applications 180~(1) (2019) 235--255.

\bibitem{khalilet.alAUV_2018}
R.~{Chertovskih}, D.~{Karamzin}, N.~T. {Khalil}, F.~L. {Pereira}, An indirect
  numerical method for a time-optimal state-constrained control problem in a
  steady two-dimensional fluid flow, in: 2018 IEEE/OES Autonomous Underwater
  Vehicle Workshop (AUV), 2018, pp. 1--6.
\newblock \href {http://dx.doi.org/10.1109/AUV.2018.8729750}
  {\path{doi:10.1109/AUV.2018.8729750}}.

\bibitem{Dubovitskii_Milyutin_1965}
A.~Y. Dubovitskii, A.~A. Milyutin, Extremum problems in the presence of
  restrictions, USSR Computational Mathematics and Mathematical Physics 5~(3)
  (1965) 1--80.

\bibitem{Halkin_1970}
H.~Halkin, A satisfactory treatment of equality and operator constraints in the
  dubovitskii-milyutin optimization formalism, Journal of optimization theory
  and applications 6~(2) (1970) 138--149.

\bibitem{Arutyunov_Tynyanskiy_1985}
A.~Arutyunov, N.~Tynyanskiy, The maximum principle in a problem with phase
  constraints, Sov. j. comput. syst. sci 23~(1) (1985) 28--35.

\bibitem{Vinter_Fereira_1994}
M.~Ferreira, R.~Vinter, When is the maximum principle for state constrained
  problems nondegenerate?, Journal of Mathematical Analysis and Applications
  187~(2) (1994) 438--467.

\bibitem{Arutyunov_Aseev_1997}
A.~V. Arutyunov, S.~M. Aseev, Investigation of the degeneracy phenomenon of the
  maximum principle for optimal control problems with state constraints, SIAM
  Journal on Control and Optimization 35~(3) (1997) 930--952.

\bibitem{Arutyunov_2000}
A.~V. Arutyunov, Optimality conditions: Abnormal and degenerate problems, Vol.
  526, Springer Science \& Business Media, 2013.

\bibitem{Vinter_2000}
R.~Vinter, Optimal control, Springer Science \& Business Media, 2010.

\bibitem{fontes_2015}
F.~A. Fontes, H.~Frankowska, Normality and nondegeneracy for optimal control
  problems with state constraints, Journal of Optimization Theory and
  Applications 166~(1) (2015) 115--136.

\bibitem{Bettiol_Khalil_Vinter_2016}
P.~Bettiol, N.~Khalil, R.~B. Vinter, {Normality of generalized Euler-Lagrange
  conditions for state constrained optimal control problems}, J. Convex Anal
  23~(1) (2016) 291--311.

\bibitem{Bryson_1969}
A.~E. Bryson, Y.-C. Ho, Applied optimal control, revised printing, 1975.

\bibitem{Jacobson}
D.~Jacobson, M.~Lele, A transformation technique for optimal control problems
  with a state variable inequality constraint, IEEE Transactions on Automatic
  Control 14~(5) (1969) 457--464.

\bibitem{Betts_1993}
J.~T. Betts, W.~P. Huffman, Path-constrained trajectory optimization using
  sparse sequential quadratic programming, Journal of Guidance, Control, and
  Dynamics 16~(1) (1993) 59--68.

\bibitem{Fabien}
B.~C. Fabien, Numerical solution of constrained optimal control problems with
  parameters, Applied Mathematics and Computation 80~(1) (1996) 43--62.

\bibitem{Maurer_2000}
C.~B{\"u}skens, H.~Maurer, {SQP-methods for solving optimal control problems
  with control and state constraints: adjoint variables, sensitivity analysis
  and real-time control}, Journal of computational and applied mathematics
  120~(1-2) (2000) 85--108.

\bibitem{Pytlak}
R.~Pytlak, Numerical methods for optimal control problems with state
  constraints, Springer, 2006.

\bibitem{Haberkorn_2011}
T.~Haberkorn, E.~Tr{\'e}lat, Convergence results for smooth regularizations of
  hybrid nonlinear optimal control problems, SIAM Journal on Control and
  Optimization 49~(4) (2011) 1498--1522.

\bibitem{Keulen}
T.~van Keulen, J.~Gillot, B.~de~Jager, M.~Steinbuch, Solution for state
  constrained optimal control problems applied to power split control for
  hybrid vehicles, Automatica 50~(1) (2014) 187--192.

\bibitem{malanowski1998sensitivity}
K.~Malanowski, H.~Maurer, Sensitivity analysis for state constrained optimal
  control problems, Discrete \& Continuous Dynamical Systems-A 4~(2) (1998)
  241--272.

\bibitem{bonnans2013shooting}
J.~F. Bonnans, The shooting approach to optimal control problems, IFAC
  Proceedings Volumes 46~(11) (2013) 281--292.

\bibitem{Filippov_1959}
A.~Filippov, On certain problems of optimal regulation, Vestn. MGU, Mat.
  Mekh~(2) (1959) 25--38.

\bibitem{nr}
W.~H. Press, S.~A. Teukolsky, W.~T. Vetterling, B.~P. Flannery, Numerical
  recipes 3rd edition: The art of scientific computing, Cambridge university
  press, 2007.

\bibitem{bvp}
H.~B. Keller, Numerical methods for two-point boundary-value problems, Courier
  Dover Publications, 2018.

\bibitem{jota}
A.~V. Arutyunov, D.~Y. Karamzin, F.~L. Pereira, {The maximum principle for
  optimal control problems with state constraints by RV Gamkrelidze:
  revisited}, Journal of Optimization Theory and Applications 149~(3) (2011)
  474--493.

\end{thebibliography}

\end{document}